\documentclass[11pt,reqno]{amsart}

\usepackage{a4wide}
\usepackage{amsmath,amssymb} 
\usepackage{bbm}
\usepackage{color}
\usepackage{tikz}
\usepackage{tkz-graph}
\usepackage{tikz-cd} 
\usepackage{dsfont}

\usepackage{bm}

\theoremstyle{plain}
\newtheorem{theorem}{Theorem}
\newtheorem{prop}[theorem]{Proposition}
\newtheorem{lemma}[theorem]{Lemma}
\newtheorem{coro}[theorem]{Corollary}
\newtheorem{fact}[theorem]{Fact}

\theoremstyle{definition}
\newtheorem{definition}[theorem]{Definition}
\newtheorem{remark}[theorem]{Remark}

\newtheorem{example}[theorem]{Example}

\newcommand{\mc}{\mathcal}

\newcommand{\ts}{\hspace{0.5pt}}

\newcommand{\exend}{\hfill $\Diamond$}

\usepackage{mathtools}
\DeclarePairedDelimiter\abs{\lvert}{\rvert}
\DeclarePairedDelimiter\norm{\lVert}{\rVert}
\DeclarePairedDelimiter{\ceil}{\lceil}{\rceil}
\DeclarePairedDelimiter{\floor}{\lfloor}{\rfloor}

\begin{document}

\title[Inflation word entropy for semi-compatible random substitutions]{
Inflation word entropy for \\[2mm]semi-compatible random substitutions
}

\author{Philipp Gohlke}
\address{Fakult\"at f\"ur Mathematik, Universit\"at Bielefeld, \newline
\hspace*{\parindent}Postfach 100131, 33501 Bielefeld, Germany}
\email{pgohlke@math.uni-bielefeld.de}

\begin{abstract}
We introduce the concept of inflation word entropy for random substitutions with a constant and primitive substitution matrix. Previous calculations of the topological entropy of such systems implicitly used this concept and established equality of topological entropy and inflation word entropy, relying on ad hoc methods. We present a unified scheme, proving that inflation word entropy and topological entropy in fact coincide. The topological entropy is approximated by a converging series of upper and lower bounds which, in many cases, lead to an analytic expression.
\end{abstract}

\keywords{Random substitutions, topological entropy, periodic points}

\subjclass[2010]{37B10, 37B40, 52C23}

\maketitle

\section{Introduction}
Random substitution systems provide a model for structures that exhibit both long-range correlations and a positive topological entropy. This combination of aperiodic order and high complexity produces new features like a generic occurrence of both pure point and absolutely continuous components in the diffraction image \cite{bss,gmaster,moll2}. Since the early exposition of random substitutions as branching processes in \cite{peyriere}, they have served as models both in physics \cite{gl} and mathematical biology \cite{koslicki,Li}. A systematic study of some of their properties was undertaken in \cite{rs}. The non-trivial topological entropy is a property that distinguishes random substitutions from usual (deterministic) substitutions which are known to have a complexity function that increases at most linearly \cite{baake}. It has been the subject of recent work to quantify the topological entropy for some specific (families of) examples of random substitutions \cite{bss,gl,gmaster,nilsson-RNMS,nilsson-sqfib}. In all of these references, the first step was to quantify the growth of the number of possible \emph{inflation words}, built from an initial letter under iterated actions of the random substitution. The concept of \emph{inflation word entropy}, introduced in this paper, accounts for this procedure. We show that for a large class of random substitutions, that we will call \emph{semi-compatible}, the inflation word entropy reproduces the value of the topological entropy. Also, we present a way to calculate the inflation word entropy efficiently, yielding a closed form expression in many cases. This result reproduces all of the known values of topological entropy mentioned above, providing much simpler proofs in some of the cases. It also allows us to work out the corresponding value for many new examples. Finally, we show that the topological entropy can be obtained from periodic words (if they exist) and connect it to some related notions for the corresponding tiling spaces.
\\ The paper is structured as follows. In Section~\ref{Sec:Setup} we introduce semi-compatible random substitutions and define the inflation word entropy for those systems. After some preliminary properties, we prove our main result in Section~\ref{Sec:main}, giving a scheme to calculate the topological entropy at the same time. Section~\ref{Sec:examples} is devoted to the discussion of examples and gives a number of criteria that allow us to simplify further the calculation of the topological entropy. We make a connection to periodic points in Section~\ref{Sec:related} and discuss a number of entropy concepts for geometric variants of the symbolic dynamical systems arising from a random substitution. This includes suspensions and other shift systems with a continuous group action.

\section{Setup and notation}
\label{Sec:Setup}

We fix a finite alphabet $\mathcal{A} = \{ a_1, \ldots, a_n\}$ of cardinality $n$ and take $\mathcal{A}^+$ to be the set of finite words in $\mc A$. The set of bi-infinite words in the alphabet $\mathcal{A}$ is given by $\mathcal{A}^{\mathbb{Z}}$. Further, we denote by $\mathcal{F} (\mc A^+)$ the set of finite (non-empty) subsets of $\mathcal{A}^+$. For $A,B \in \mc F(\mc A^+)$ we define a concatenation of sets as $AB := \{ uv \mid u \in A, v \in B\}$, where $uv$ denotes the standard concatenation of words. The cardinality of a set $A \in \mc F(\mc A^+)$ is denoted by $\# A$. A subword of a word $u=u_1 \cdots u_m$ is any word of the form $u_{[k,\ell]} := u_k \cdots u_{\ell}$, with $1 \leqslant k \leqslant \ell \leqslant m$. For every $u,v \in \mc A^+$, we write $v \triangleleft u$ if $v$ is a subword of $u$. Further, we let $\abs{u}$ denote the (symbolic) length of the word $u$ and $\abs{u}_v$ the number of occurences of $v$ in $u$ as a subword. If any of these values is the same for all words $u$ in a set $A \in \mc F(\mc A^+)$, we define $\abs{A}_v := \abs{u}_v$ and $\abs{A}_{(\ell)} := \abs{u}$, for arbitrary $u \in A$. The subscript $(\ell)$ serves as a reminder that the latter should not be mistaken for the cardinality of the set $A$. The \emph{Abelianisation} of a word $u \in \mc A^{+}$ is an $n$-component vector $\Phi(u)$, with $\Phi(u)_i = \abs{u}_{a_i}$, for all $1 \leqslant i \leqslant n$. As before, we can extend $ \Phi$ to sets of words that share a common Abelianisation.
\\A random substitution generalizes the notion of a substitution on the alphabet $\mc A$ by allowing that a letter $a \in \mc A$ is mapped to different words with predetermined probabilities. For the present, purely combinatorial purpose, we avoid specifying the probabilities. 
\begin{definition}
A \emph{random substitution} on the alphabet $\mc A$ is a map $\vartheta \colon \mc A \to \mc F(\mc A^+)$. It is extended to $\mc A^+$ via concatenation of sets
\[
\vartheta \colon \mc A^+ \ni u = u_1 \cdots u_m \, \mapsto \, \vartheta(u_1) \cdots \vartheta(u_m)
\]
and to $\mc F(\mc A^+)$ via taking unions
\[
\vartheta \colon \mc F(\mc A^+) \ni A \, \mapsto \, \bigcup_{u \in A} \vartheta(u),
\]
where the union is not necessarily disjoint. A random substitution $\vartheta$ is called \emph{semi-compatible} if for all $a \in \mc A$ we have that $u,v \in \vartheta(a)$ implies $\bm \Phi(u) = \bm \Phi(v)$.
\end{definition}
The definition of a random substitution given above is called a \emph{multi-valued substitution} in \cite{gmaster}. See the same reference for a discussion of how this concept is related to the notion of random substitutions as defined in \cite{gs,moll,rs}. 
\\Note that for a semi-compatible random substitution, both $\abs{\vartheta^m(v)}_a$ and $\abs{\vartheta^m(v)}_{(\ell)}$ are well-defined for all $a \in \mc A$, $v \in \mc A^+$ and all powers $m \in \mathbb{N}$ of the substitution. A word $v \in \vartheta^m(a)$, for some $a \in \mc A$, is called a (level-$m$) \emph{inflation word} (starting from $a$).

\begin{example}
The random Fibonacci substitution on $\mc A = \{a,b\}$ is defined via $\vartheta \colon a \mapsto \{ab,ba\}$ and $b \mapsto \{a\}$. Here, $\vartheta$ is semi-compatible because $\vartheta(b)$ is a singleton-set and all words in $\vartheta(a)$ have the same number of letters $a$ and $b$ appearing; more precisely, $\Phi(ab) = \Phi(ba) = (1,1)^{\intercal}$.
\end{example}

\begin{definition}
Let $\vartheta$ be a semi-compatible random substitution on $\mc A = \{a_1,\ldots,a_n\}$. The associated \emph{substitution matrix} $M$ is defined via $M_{ij} = \abs{\vartheta(a_j)}_{a_i}$ for all $1 \leqslant i,j \leqslant n$. We call $\vartheta$ \emph{primitive} if $M$ is a primitive matrix. In this case, we denote by $\lambda$ the Perron--Frobenius (PF) eigenvalue of $M$ and write $\bm R$ and $\bm L$ for the right and left PF eigenvector, respectively. The normalisation is chosen as $\norm{\bm R}_1 = 1 = \bm L^\intercal \bm R$.
\end{definition}

\begin{definition}
The \emph{language} of a random substitution $\vartheta$ is given by all words that appear as a subword of some inflation word. More precisely,
\[
\mc L = \{ v \in \mc A^+ \mid v \triangleleft u \in \vartheta^m(a), \mbox{ for some } a \in \mc A, m \in \mathbb{N}, u \in \mc A^+ \}.
\]
Words in $\mc L$ are called \emph{legal}. The set of legal words of a given length $\ell \in \mathbb{N}$, is denoted by $\mc L_{\ell} = \{ v \in \mc L \mid \abs{v} = \ell \}$.
\end{definition}

To a given language $\mc L$, there is an associated shift-dynamical system $(\mathbb{X},\sigma)$. Here,
\[
\mathbb{X} = \{ x \in \mc A^{\mathbb{Z}} \mid x_{[k,\ell]} \in \mc L, \mbox{ for all } k\leqslant \ell \in \mathbb{Z} \}
\]
is closed, and thus compact, as a subset of $\mc A^{\mathbb{Z}}$ (endowed with the product topology) and $\sigma \colon x \mapsto \sigma(x)$, with $\sigma(x)_i = x_{i+1}$ denotes the left shift on $\mathbb{X}$.
The corresponding \emph{topological entropy} (compare \cite{adler}) can be expressed in terms of the language as
\[
s \, = \, \lim_{\ell \rightarrow \infty} \frac{1}{\ell} \log(\# \mc L_{\ell}).
\] 
In the following, we fix $\vartheta$ to be a primitive, semi-compatible random substitution on the alphabet $\mc A = \{a_1,\ldots, a_n\}$. For notational convenience, let us define $\bm{\ell}_m = (\ell_{m,1},\ldots,\ell_{m,n})^\intercal$ and $\bm{q}_m = (q_{m,1},\ldots,q_{m,n})^\intercal$, where 
\begin{align*}
\ell_{m,i} \, = \, \abs{\vartheta^m(a_i)}_{(\ell)} \quad \mbox{and} \quad
q_{m,i} \, = \, \log (\# \vartheta^m(a_i) ),
\end{align*} 
for all $1 \leqslant i \leqslant n$ and $m \in \mathbb{N}$. 
\begin{definition}
The \emph{upper and lower inflation word entropy of type $i$} for $1 \leqslant i \leqslant n$ are given by
\begin{align*}
\underline{s}^I_i & \,= \, \liminf_{m \rightarrow \infty} \frac{q_{m,i}}{\ell_{m,i}}  \, = \, \liminf_{m \rightarrow \infty} \frac{1}{\abs{\vartheta^m(a_i)}_{(\ell)}} \log (\# \vartheta^m(a_i)),
\\\overline{s}^I_i & \,= \, \limsup_{m \rightarrow \infty} \frac{q_{m,i}}{\ell_{m,i}} \, = \, \limsup_{m \rightarrow \infty} \frac{1}{\abs{\vartheta^m(a_i)}_{(\ell)}} \log (\#\vartheta^m(a_i)).
\end{align*} 
\end{definition}
One of the main results of this paper will be that the limits in the above expressions exist and that they are independent of $i$. This will justify us speaking of the \emph{inflation word entropy} $s^I$. Since $\vartheta^m(a_i) \subset \mc L_{\ell_{m,i}}$, it is immediate from the definition that the inflation word entropy is a lower bound for the topological entropy $s$.

The following fact determines the length of arbitrarily large inflation words by linear algebra. It follows immediately from the definition of the substitution matrix $M$; compare \cite[Ch.~4]{baake} and \cite[Ch.~5.3]{queffelec}.
\begin{fact}
\label{Fact:l-iteration}
For every\/ $m \in \mathbb{N}$, we have that\/
$
\bm{\ell}_{m}^\intercal = \bm{1}^\intercal M^m,
$
where\/ $\bm{1} = (1,\ldots,1)$ denotes the\/ $n$-dimensional vector with identical entries\/ $1$. In particular,\/ $\bm{\ell}_{m}^\intercal = \bm{\ell}_{m-1}^\intercal M$.
\end{fact}

There are two special cases for the inflation word structure of $\vartheta$ that deserve to be named since they mark the boundary cases in the calculations to come.
\begin{definition}
The random substitution $\vartheta$ is said to satisfy the \emph{identical set condition} if
\begin{equation}
\label{Eq:lower-bound-condition}
u,v \in \vartheta(a_i) \implies \vartheta^m(u) = \vartheta^m(v),
\end{equation}
for all\/ $1 \leqslant i \leqslant n$ and $m \in \mathbb{N}$. It is said to satisfy the \emph{disjoint set condition} if
\begin{equation}
\label{Eq:upper-bound-condition}
u, v \in \vartheta(a_i), u\neq v \implies \vartheta^m(u) \cap \vartheta^m(v) = \varnothing,
\end{equation}
for all\/ $1 \leqslant i \leqslant n$ and $m \in \mathbb{N}$.
\end{definition}

\begin{lemma}
\label{Lemma:q-bounds}
For every\/ $m \in \mathbb{N}$, we have
\[
\bm{q}_m^\intercal M \, \leqslant \, \bm{q}_{m+1}^\intercal 
\, \leqslant \, \bm{q}_m^\intercal M + \bm{q}_1^\intercal,
\]
where the inequalities are to be understood element-wise. The lower bound is an equality if and only if \eqref{Eq:lower-bound-condition} holds and the upper bound is an equality if and only if
\eqref{Eq:upper-bound-condition} holds for all $1\leqslant i \leqslant n$.
\end{lemma}

\begin{proof}
Let $m \in \mathbb{N}$ and $1 \leqslant i \leqslant n$. Then,
\begin{align}
\label{Eq:set-iteration}
\vartheta^{m+1}(a_i) \, = \, \vartheta^m(\vartheta(a_i)) 
\, = \, \bigcup_{u \in \vartheta(a_i)} \vartheta^m(u). 
\end{align}
For the moment, fix an arbitrary $u \in \vartheta(a_i)$ and suppose $\abs{u} = r$. Then, by definition, $\vartheta^m(u) = \vartheta^m(u_1)\cdots\vartheta^m(u_r)$ and since all the words in $\vartheta^m(u_k)$, for a fixed $1 \leqslant k \leqslant r$, have the same length, we find
\[
\# \vartheta^m(u) \, = \, \prod_{k=1}^r \# \vartheta^m(u_k) 
\, = \, \prod_{j=1}^n (\# \vartheta^m(a_j))^{\abs{u}_{a_j}}.
\] 
Due to the semi-compatibility condition,
$\abs{u}_{a_j} = \abs{\vartheta(a_i)}_{a_j} = M_{ji}$, for all $1 \leqslant j \leqslant n$. Thus,
\[
\# \vartheta^m(u) \, = \, \prod_{j=1}^n (\# \vartheta^m(a_j))^{M_{ji}}.
\]
Note that this is independent of the choice of $u \in \vartheta(a_i)$.
Taking cardinalities in \eqref{Eq:set-iteration} therefore yields
\[
\prod_{j=1}^n (\# \vartheta^m(a_j))^{M_{ji}} 
\, \leqslant \, \# \vartheta^{m+1}(a_i) 
\, \leqslant \, (\# \vartheta(a_i)) \prod_{j=1}^n (\# \vartheta^m(a_j))^{M_{ji}} ,
\]
where the lower bound is an equality if and only if all the sets in the union in \eqref{Eq:set-iteration} coincide and the upper bound is an equality if and only if the union is disjoint. Taking the logarithm gives the desired relation.
\end{proof}

\begin{remark}
More generally, we can show that for $1\leqslant k \leqslant m$,
\[
\bm q_k^{\intercal} M^{m-k} \, \leqslant \, \bm q_m^{\intercal} 
\, \leqslant \, \bm q_k^{\intercal} M^{m-k} + \bm q^{\intercal}_{m-k}.
\]
This follows by splitting $\vartheta^m(a_i) = \vartheta^k(\vartheta^{m-k}(a_i))$ in place of \eqref{Eq:set-iteration} and then following the same steps as in the proof above. It also leads to slightly different, although related, conditions for the realization of the upper or lower bound. \exend
\end{remark}

\section{Main results}
\label{Sec:main}
Our strategy to prove that the inflation word entropy is well defined and coincides with the topological entropy is to establish a sequence of lower and upper bounds for both quantities which eventually narrow down the set of possible values to a single point.

The first step in this direction is the following result.

\begin{prop}
\label{Prop:inflation-entropy-bounds}
The upper and lower inflation word entropy are bounded by
\begin{equation}
\frac{1}{\lambda}\bm{q}_1^\intercal \bm R
 \, \leqslant \, \underline{s}_i^I \, \leqslant \, \overline{s}_i^I 
 \, \leqslant \, \frac{1}{\lambda-1}\bm{q}_1^\intercal \bm R,
\end{equation}
for all\/ $1 \leqslant i \leqslant n$. The lower bound is an equality if the identical set condition is satisfied and the upper bound is an equality if the disjoint set condition is satisfied. For the lower bound, we additionally have that
\begin{equation}
\label{Eq:lower-bound-sequence}
\frac{1}{\lambda^r}\bm{q}_r^\intercal \bm R \leqslant \underline{s}_i^I ,
\end{equation}
for all\/ $r \in \mathbb{N}$, where the lower bound is a monotonically increasing function in\/ $r$.
\end{prop}

\begin{proof}
This is basically a direct consequence of Lemma~\ref{Lemma:q-bounds} by an iterative application. More precisely, for $m \geqslant 1$, we find, focusing first on the lower bound,
\[
q_{m+1,i} \, \geqslant \, \left( \bm q_r^\intercal M^{m+1-r} \right)_i
\]
for every $1 \leqslant i \leqslant n$ and $m \in \mathbb{N}$ that is larger than a fixed $r \in \mathbb{N}$. Dividing both sides by $\ell_{m+1,i}$ and taking the $\liminf$ yields
\[
\underline{s}_i^I \, \geqslant \, \liminf_{m \rightarrow \infty} \frac{1}{\lambda^r} \frac{\lambda^{m+1}}{\ell_{m+1,i}} \left( \bm q_r^\intercal \frac{M^{m+1-r}}{\lambda^{m+1-r}} \right)_i 
\, = \, \frac{1}{\lambda^r} (L_i)^{-1} \bm q_r^\intercal \bm R L_i 
\, = \, \frac{1}{\lambda^r}\bm{q}_r^\intercal \bm R.
\]
The penultimate step follows from standard PF theory and the assumption that $M$ is primitive, implying
\[
\lim_{m \rightarrow \infty} \frac{1}{\lambda^m} M^m 
\, = \, \bm R \bm L^\intercal
\]
and 
\begin{equation} \label{Eq:asymptotic-length}
\lim_{m \rightarrow \infty} \frac{1}{\lambda^m} \ell_{m,i}
\, = \, \lim_{m \rightarrow \infty} \frac{1}{\lambda^m} \left( \bm 1^\intercal M^m \right)_i 
\, = \, \bm 1^\intercal \bm R L_i = L_i,
\end{equation}
where we have used Lemma~\ref{Fact:l-iteration} and the normalisation condition on the right eigenvector. In order to show monotonicity, note that 
\[
\frac{1}{\lambda^{r+1}} \bm q_{r+1}^\intercal \bm R
\, \geqslant \, \frac{1}{\lambda^{r+1}} \bm q_{r}^\intercal M \bm R \, = \, \frac{1}{\lambda^r} \bm q_r^\intercal \bm R.
\]

For the upper bound, we proceed similarly. First,
\[
q_{m+1,i} \leqslant \left( \bm q_1^\intercal \sum_{k=0}^{m} M^k \right)_i,
\]
such that 
\[
\overline{s}_i^I \leqslant \limsup_{m \rightarrow \infty} \frac{1}{\lambda} \frac{\lambda^{m+1}}{\ell_{m+1,i}} \left( \bm q_1^\intercal \frac{1}{\lambda^m} \sum_{k=0}^{m} M^k \right)_i
 = \frac{1}{\lambda} \left(1-\frac{1}{\lambda} \right)^{-1} \bm q_1^\intercal \bm R = \frac{1}{\lambda - 1} \bm q_1^\intercal \bm R.
\]
Note that, in the second step, we have made use of PF theory once more to conclude that
\[
\limsup_{m \rightarrow \infty} \frac{1}{\lambda^m} \sum_{k=0}^m M^k 
\, = \, \limsup_{m \rightarrow \infty}  \sum_{j=0}^m \frac{1}{\lambda^{j}} \frac{M^{m-j}}{\lambda^{m-j}}
\, = \,  \sum_{j = 0}^{\infty} \frac{1}{\lambda^{j}} \bm R \bm L^\intercal
 = \left( 1 - \frac{1}{\lambda} \right) ^{-1} \bm R \bm L^\intercal.
\]
The claim on the sufficient condition for the realization of the lower or upper bound is immediate from Lemma~\ref{Lemma:q-bounds} and the above calculation.
\end{proof}

As we remarked earlier, we clearly have $\overline{s}_i^I \leqslant s$. Our next step will be to bound $s$ by the same upper bound that was given in Proposition~\ref{Prop:inflation-entropy-bounds} using an independent argument.
\begin{prop}
\label{Prop:entropy-upper-bound}
The topological entropy associated with $\vartheta$ satisfies
\[
s \, \leqslant \, \frac{1}{\lambda-1} \bm q_1^\intercal \bm R.
\]
\end{prop}

We proceed by a number of steps. First, we observe that semi-compatibility still guarantees the existence of uniform letter frequencies. Conceptionally, this will be at the heart of the proof of Proposition~\ref{Prop:entropy-upper-bound}.

\begin{prop}
The language\/ $\mc L$, corresponding to\/ $\vartheta$, exhibits uniform existence of letter frequencies in the following sense. For all\/ $\varepsilon > 0$ there is a length\/ $\ell$ such that for all\/ $u \in \mc L$ with\/ $\abs{u} \geqslant \ell$, we have
\begin{equation}
\label{Eq:letter-frequencies}
\left\vert\frac{\abs{u}_{a_i}}{\abs{u}} - R_i \right\vert 
\, < \, \varepsilon.
\end{equation}
\end{prop}
\begin{proof}[Sketch of proof]
This is a well-known fact for primitive deterministic substitutions. The reason that it still holds for primitive semi-compatible random substitutions is that this is a property that basically relies on the Abelianisations of the inflation words only, which are fixed by the substitution matrix $M$. More specifically, compare the proof of \cite[Thm.~5.6]{queffelec}. This carries over to our situation almost verbatim, if restricted to \emph{letters} instead of general subwords (note that the construction no longer carries over for general subwords since this would involve induced substitutions which are no longer semi-compatible in the random case). 
\end{proof}

With this tool at hand, we are in the position to show that for large enough words, applying the substitution expands the words by a factor $\lambda$, up to a small deviation.

\begin{lemma}
\label{Lemma:superset-for-inflation-words}
Let\/ $I_N = \vartheta(\mc L) \cap \mc L_N$. For all\/ $\varepsilon >0$, there is an\/ $N_0 \in \mathbb{N}$ such that, for all\/ $N \geqslant N_0$,
\begin{equation}
\label{Eq:superset-for-inflation-words}
I_N \, \subset \, \bigcup_{m = \ceil{N/(\lambda + \varepsilon n K)}}^{ \floor{N/(\lambda - \varepsilon n K)}} \vartheta(\mc L_m) ,
\end{equation}
where\/ $K= \max_{1 \leqslant i \leqslant n} \ell_{1,i}$ and\/ $n = \# \mc A$, assuming\/ $\varepsilon$ is small enough that\/ $\lambda > \varepsilon n K$.
\end{lemma}
\begin{proof}
Let $\varepsilon >0$, choose $\ell \in \mathbb{N}$ such that \eqref{Eq:letter-frequencies} holds and let $m \geqslant \ell$. Then, for $u \in \mc L_m$, we find that $\abs{u}_{a_i} \in ( (R_i - \varepsilon)m, (R_i + \varepsilon)m)$ and thereby
\[
\abs{\vartheta(u)}_{(\ell)} \, = \, \sum_{i=1}^n \abs{u}_{a_i} \abs{\vartheta(a_i)}_{(\ell)} 
\, < \, m \sum_{i=1}^n R_i \ell_{1,i} + \varepsilon m \sum_{i=1}^n \ell_{1,i}
\, \leqslant \, m \bm \ell_1^\intercal \bm R + \varepsilon m n K
\, = \, m (\lambda + \varepsilon n K),
\]
where in the last step we have made use of $\bm \ell_1^\intercal \bm R = \mathbf{1}^\intercal M \bm R = \lambda$. Analogously,
\[
\abs{\vartheta(u)}_{(\ell)} \, > \, m (\lambda - \varepsilon n K).
\]
Now choose $N \geqslant N_0 = \ceil{\ell (\lambda + \epsilon n K)}$ and $w \in I_N$ with $w \in \vartheta(u)$ and $\abs{u} = m$. Then, by the bounds above 
\[
 m(\lambda - \varepsilon n K) \, < \, \abs{\vartheta(u)}_{(\ell)}  \, = \, N  \, < \, m (\lambda + \varepsilon n K),
\]
or, equivalently
\[
\frac{N}{\lambda + \varepsilon n K} \, < \, m \, < \, \frac{N}{\lambda - \varepsilon n K},
\]
which establishes the claim.
\end{proof}

\begin{coro}
\label{Corr:inflation-word-cardinality}
In the situation of Lemma~\textnormal{\ref{Lemma:superset-for-inflation-words}},  let\/ $\lambda^{\pm}_\varepsilon = \lambda \pm \varepsilon n K$. Then, for\/ $N \geqslant N_0$,
\[
\# I_N \, \leqslant \, N \left(\frac{1}{\lambda^-_\varepsilon} - \frac{1}{\lambda^+_\varepsilon} \right) \bigl(\# \mc L_{ \floor{ N/\lambda^-_\epsilon} } \bigr) \prod_{i = 1}^{n} \left( \# \vartheta(a_i) \right)^{ (R_i + \varepsilon) N/\lambda^-_\varepsilon}.
\]
\end{coro}

\begin{proof}
Taking the cardinality of \eqref{Eq:superset-for-inflation-words}, we find
\[
\# I_N \, \leqslant \, \sum_{m = \ceil{N/\lambda^+_\varepsilon}}^{\floor{N/\lambda^-_\varepsilon}} \# \vartheta(\mc L_m) 
\, \leqslant \, \sum_{m = \ceil{N/\lambda^+_\varepsilon}}^{\floor{N/\lambda^-_\varepsilon}} \sum_{u \in \mc L_m} \# \vartheta(u) .
\]
By the construction in the proof above, we have for $u \in \mc L_m$,  with $N/\lambda^+_\varepsilon \leqslant m \leqslant N/\lambda^-_\varepsilon$,
\[
\# \vartheta(u) \, = \, \# \vartheta(u_1) \cdots \# \vartheta(u_m) 
\, = \, \prod_{i=1}^n \left( \# \vartheta(a_i) \right)^{\abs{u}_{a_i}} 
\, \leqslant \, \prod_{i=1}^n \left( \# \vartheta(a_i) \right)^{(R_i + \epsilon) m},
\]
which is monotonically increasing in $m$. Since also $\# \mc L_m$ is increasing in $m$, the claim follows.
\end{proof}

Finally, let us prove the upper bound for the topological entropy.

\begin{proof}[Proof of Proposition~\textnormal{\ref{Prop:entropy-upper-bound}}]
Suppose $v \in \mc L_N$. Then, there exists a $u \in \mc L$ such that $v \triangleleft v' \in \vartheta(u)$ for some $v' \in \mc L$ with $N \leqslant \abs{v'} \leqslant N+2(K-1)$, with $K$ as in Lemma~\ref{Lemma:superset-for-inflation-words}. To account for the different possible positions of $v$ within $v'$, let us define the sets
$\mc F_p^L(N) = \{ w_{[p,p+N-1]} \mid w \in I_L \}$, where we suppose that $N \leqslant L$ and consider $p$ within the range $1\leqslant p \leqslant L-N+1$. Obviously,
\[
\mc L_N \, \subset \, \bigcup_{L=N}^{N+2K} \bigcup_{p=1}^{L-N+1} \mc F_p^L(N).
\]
From the definition, it is clear that $\# \mc F_p^L(N) \leqslant \# I_L$ for any $N$, $p$ and $L$ in their corresponding ranges. In particular, this bound is independent of $p$. Thereby,
\begin{align*}
\# \mc L_N & \, \leqslant \, \sum_{L=N}^{N+2K} (L-N+1) (\# I_L) \\
& \, \leqslant \,  (2K +1)^2 (N+2K) \left(\frac{1}{\lambda^-_\varepsilon} - \frac{1}{\lambda^+_\varepsilon} \right) \bigl( \# \mc L_{ \floor{ (N+2K)/\lambda^-_\epsilon} } \bigr) \prod_{i = 1}^{n} \left( \# \vartheta(a_i) \right)^{ (R_i + \varepsilon) (N+2K)/\lambda^-_\varepsilon},
\end{align*}
where in the last step we have made use of Corollary~\ref{Corr:inflation-word-cardinality}, together with the monotonicity of the functions occurring therein. We note that the first three factors of the right hand side in the last expression exhibit only linear growth in $N$ and therefore vanish under an application of $\frac{1}{N}\log(\cdot)$, in the limit of large $N$. Consequently,
\begin{align*}
s & \, = \, \lim_{N \rightarrow \infty} \frac{1}{N} \log \left( \# \mc L_N \right) \\
&\, \leqslant \, \lim_{N \rightarrow \infty}\frac{1}{N } \log \bigl( \# \mc L_{ \floor{ (N+2K)/\lambda^-_\epsilon} } \bigr) + \frac{N+2K}{N \lambda^-_\varepsilon} \sum_{i=1}^{n} (R_i + \varepsilon) \log (\# \vartheta(a_i))
\\ & \, = \, \frac{1}{\lambda^-_\varepsilon} s + \frac{1}{\lambda^-_\varepsilon} \bm q_1^\intercal \bm R + \frac{\varepsilon}{\lambda^-_{\varepsilon}} \sum_{i=1}^{n} q_{1,i} \, \xrightarrow{ \ts \varepsilon \rightarrow 0 \ts } \, \frac{1}{\lambda} s + \frac{1}{\lambda} \bm q_1^\intercal \bm R.
\end{align*}
This yields
\[
s \, \leqslant \, \left( 1 - \frac{1}{\lambda} \right) ^{-1} \frac{1}{\lambda} \bm q_1^\intercal \bm R 
\, = \, \frac{1}{\lambda -1} \bm q_1^\intercal \bm R,
\]
as desired.
\end{proof}

We summarise our main findings in the following result.
\begin{theorem}
\label{Thm:entropy=inflation-entropy}
Suppose\/ $\vartheta$ is a primitive semi-compatible random substitution on the alphabet\/ $\mc A = \{a_1,\ldots,a_n\}$ with entropy\/ $s$ and lower (upper) inflation word entropies\/ $\underline{s}_i^I$ ($\overline{s}_i^I$). Then, for\/ $1\leqslant i \leqslant n$ and any\/ $m \in \mathbb{N}$, we have
\begin{equation}
\label{Eq:Thm-entropy-bounds}
\frac{1}{\lambda^m} \bm q_m^\intercal \bm R
\, \leqslant \, \underline{s}_i^I 
\, \leqslant \, \overline{s}_i^I
\, \leqslant \, s \, \leqslant \, \frac{1}{\lambda^m -1} \bm q_m^\intercal \bm R.
\end{equation}
In particular, the inflation word entropy\/ \[s^I = \lim_{m \rightarrow \infty} \frac{1}{\ell_{m,i}} \log(\#\vartheta^m(a_i))\]
is well-defined, independent of\/ $i$ and equals the topological entropy\/ $s$. Both can be calculated as
\begin{equation}
\label{Eq:Thm-entropy-calc}
s \, = \, s^I \, = \, \lim_{m \rightarrow \infty} \frac{1}{\lambda^m} \bm q_m^\intercal \bm R \, = \, \sup_{m \in \mathbb{N}} \frac{1}{\lambda^m} \bm q_m^\intercal \bm R.
\end{equation}
\end{theorem}

\begin{proof}
The only statement in \eqref{Eq:Thm-entropy-bounds}, which is not immediate from Proposition~\ref{Prop:inflation-entropy-bounds} and Proposition~\ref{Prop:entropy-upper-bound}, is maybe the upper bound for $s$ for arbitrary values of $m \in \mathbb{N}$. However, this follows readily from Proposition~\ref{Prop:entropy-upper-bound}, applied to the substitution $\vartheta^m$ (which is itself clearly primitive and semi-compatible) and the observation that $q_{m,i}(\vartheta) = \log (\# \vartheta^m(a_i)) = q_{1,i}(\vartheta^m)$, combined with the fact that the topological entropy $s$ is identical for any power of $\vartheta$. 
\\ This establishes that $s^I$ is well-defined via
\[
1 \, \leqslant \, \frac{\overline{s}_i^I}{\underline{s}_i^I} 
\, \leqslant \, \frac{\lambda^m}{\lambda^m -1} \, \xrightarrow{\ts m \rightarrow \infty \ts} \, 1.
\]
The fact that $s = s^I$ and the formula for its calculation in \eqref{Eq:Thm-entropy-calc} follow similarly.
\end{proof}

\section{Examples and Applications}
\label{Sec:examples}
The procedures presented in the last section might raise the hope to find a closed form expression for the entropy of any primitive semi-compatible random substitutions. In general, the difficulty lies in quantifying the overlaps of sets of the form $\vartheta^m(u)$, $u \in \vartheta(a_i)$, in case they are non-trivial and thereby lie strictly in between \eqref{Eq:lower-bound-condition} and \eqref{Eq:upper-bound-condition}. Usually, one works out inductive relations for these intersections which are particular for the substitution at hand---compare \cite{bss, gmaster, nilsson-RNMS, nilsson-sqfib, nilsson-rfib}. However, if one of the limiting cases holds, we \emph{do} get a closed formula for the topological entropy.

\begin{coro}
The identical set condition implies
\[
s \, = \, \frac{1}{\lambda} \bm q_1^\intercal \bm R,
\]
while the disjoint set condition is sufficient for
\[
s  \, = \, \frac{1}{\lambda - 1} \bm q_1^\intercal \bm R.
\]
\end{coro}

\begin{proof}
This is an immediate consequence of Proposition~\ref{Prop:inflation-entropy-bounds} and Theorem~\ref{Thm:entropy=inflation-entropy}.
\end{proof}

Clearly, a sufficient criterion for the identical set condition is that $\vartheta(a) = \vartheta(b)$ for all $a,b \in \mc A$. One such substitution was considered in \cite{rs}.

\begin{example}
Let $\vartheta \colon a \mapsto \{ab,ba\}, \, b \mapsto \{ab,ba\}$ be a random substitution with $\lambda = 2$ and $\bm R = (1/2,1/2)^\intercal$. It was shown in \cite{rs} that the subshift constructed from this random substitution is in fact a sofic shift. Since $\vartheta(a) = \vartheta(b)$, we find
\[
s \, = \, \frac{1}{2} (\log(2),\log(2)) \bm R \, = \, \frac{1}{2} \log(2)
\] 
for the topological entropy. \exend
\end{example}

In the case of constant length random substitutions, there is an easy sufficient criterion to ensure the disjoint set condition.

\begin{coro}
\label{Corr:constant-length}
In the situation above, assume in addition that\/ $\vartheta$ is a constant-length substitution. That is, there exists a length\/ $k \in \mathbb{N}$ such that\/ $\abs{\vartheta(a)}_{(\ell)} = k$, for all\/ $a \in \mc A$. If, in addition,\/ $\vartheta$ satisfies the \emph{disjoint inflation set condition}
\begin{equation}
\vartheta(a) \cap \vartheta(b) \, = \, \varnothing,
\end{equation}
for all\/ $a,b \in \mc A$ with\/ $a \neq b$, then the entropy is given by the upper bound
\[
s \, = \, \frac{1}{\lambda -1} \bm q_1^\intercal \bm R.
\]
\end{coro}

\begin{proof}
It clearly suffices to show that the disjoint inflation set condition implies the disjoint set condition in the constant-length setting. 
\\Let $a,b \in \mc A$, with $a\neq b$. First, it follows by induction that $\vartheta^m(a) \cap \vartheta^m(b) = \varnothing$ for all $m \in \mathbb{N}$: Suppose it is true for all $m$ up to $m_0 \in \mathbb{N}$. With the aim of establishing a contradiction, suppose further that $w \in \vartheta^{m_0+1}(a) \cap \vartheta^{m_0+1}(b)$. Then, there is $u \in \vartheta^{m_0}(a)$ and $v \in \vartheta^{m_0}(b)$ such that $w \in \vartheta(u) \cap \vartheta(v)$. Since $\vartheta^{m_0}(a) \cap \vartheta^{m_0}(b) = \varnothing$ by the induction assumption, $u \neq v$, so there exists a position $j$ such that $u_j \neq v_j$ (recall that $\abs{u} = \abs{v}$ because $\vartheta$ is constant-length). Thus, $w \in \vartheta(u) \cap \vartheta(v)$ implies that
\[
w_{[(j-1)k + 1,j k]} \in \vartheta(u_j) \cap \vartheta(v_j)
\, = \, \varnothing
\]
by the constant-length condition, giving the desired contradiction.
\\Next, let $1 \leqslant i \leqslant n$ and $u,v \in \vartheta(a_i)$, with $u \neq v$. Then, $u_j \neq v_j$ for some $1 \leqslant j \leqslant \abs{u} = \abs{v}$. Suppose there is $w \in \vartheta^m(u) \cap \vartheta^m(v)$, for some $m \in \mathbb{N}$. Since each word in $\vartheta^m(a)$, $a \in \mc A$ has length $k^m$, this would imply
\[
w_{[(j-1)k^m + 1, jk^m]} \in \vartheta^m(u_j) \cap \vartheta^m(v_j) 
\, = \, \varnothing,
\]
leading to a contradiction. Consequently $\vartheta^m(u) \cap \vartheta^m(v) = \varnothing$.
\end{proof}

Let us apply our results to a number of `test cases' for which the topological entropy has been calculated in previous work. In particular, we will look at random variants of the well-known Fibonacci, Period Doubling and Thue--Morse substitutions.

\begin{example}[Random Period Doubling]
$\vartheta_{RPD} \colon a \mapsto \{ab,ba\} ,\, b \mapsto \{aa \}$, with data $\lambda = 2$, $\bm R = (2/3,1/3)^\intercal$ and $\bm q_1 = (\log(2),0)^\intercal$. As this is a constant-length substitution satisfying the disjoint inflation set condition, we can apply Corollary~\ref{Corr:constant-length} and obtain
\[
s_{RPD} \, = \, \frac{1}{\lambda-1} \bm q_1^\intercal \bm R 
\, = \, \frac{2}{3} \log(2).
\]
This coincides with the value computed in \cite{bss}. \exend
\end{example}

\begin{example}[Random Thue--Morse]
$\vartheta_{RTM} \colon a \mapsto \{ab,ba\},\, b \mapsto \{ba\} $, with data $\lambda = 2$, $\bm R = (1/2,1/2)^\intercal$, $\bm q_1^\intercal = (\log(2),0)$. This substitution is also constant-length, but no longer satisfies the disjoint inflation set condition. Indeed,
\[
0.1733 \, \approx \, \frac{1}{4} \log(2)
 \, = \, \frac{1}{\lambda} \bm q_1^\intercal \bm R \, < \, s_{RTM} 
 \, < \, \frac{1}{\lambda - 1} \bm q_1^\intercal \bm R
\, = \, \frac{1}{2} \log(2) \approx 0.3466. 
\] 
The numerical value of $s_{RTM}$ was computed in \cite{gmaster} to be \[ s_{RTM} \, \approx \, 0.253917.\]
We can, of course, improve our bounds by going to higher powers. For this, it is useful to establish inductive relations between the sets of inflation words. Concretely, we observe that $\vartheta^m(b) \subset \vartheta^m(a)$ for all $m \in \mathbb{N}$, which yields 
\begin{align*}
\# \vartheta^{m+1}(a) &\, = \, \# \left( \vartheta^m(a) \vartheta^m(b) \cup  \vartheta^m(b) \vartheta^m(a) \right) 
 \, = \, 2 (\# \vartheta^m(a)) (\#\vartheta^m(b)) - (\# \vartheta^m(b))^2,
\end{align*}
making use of the fact that $\vartheta^m(a) \vartheta^m(b) \cap  \vartheta^m(b) \vartheta^m(a) = \vartheta^m(b) \vartheta^m(b)$. Also, 
\[
\# \vartheta^{m+1}(b) \, = \, (\# \vartheta^m(b)) (\# \vartheta^m(a)).
\]
This gives a scheme to compute $\bm q_m$ for arbitrarily large numbers $m \in \mathbb{N}$ at relatively low computational cost (as compared to naively counting the cardinalities of inflation word sets). For example, we obtain for $m = 5$,
\[
0.25177 \approx \frac{1}{64} \log(9953280) = \frac{1}{\lambda^5} \bm q_5^\intercal \bm R < s_{RTM} < \frac{1}{\lambda^5-1} \bm q_5^\intercal \bm R = \frac{1}{62} \log(9953280) \approx 0.25989,
\]
reproducing the first two valid digits. \exend
\end{example}

\begin{example}[Random Fibonacci] $\vartheta_{RF}\colon a \mapsto \{ab,ba\},\, b \mapsto \{a\}$, with data $\lambda = \tau$ the golden ration, $\bm R = \frac{1}{1+\tau} (\tau, 1 )^\intercal$, $\bm q_1^\intercal = (\log(2),0)$. Note that, due to the small inflation factor, the convergence rate of the lower and upper bounds to the real value of the entropy, given by
\begin{equation}\label{Eq:s_RF}
s_{RF} \, = \, \sum_{m=2}^{\infty} \frac{\log(m)}{\tau^{m+2}} 
\, \approx \, 0.444399,
\end{equation}
will be relatively poor. The exact value for the entropy was worked out in \cite{gl,moll,nilsson-rfib}. Concretely,
\[
0.265 \, \approx \, \frac{1}{\tau^2} \log(2)
\, = \, \frac{1}{\tau} \bm q_1^\intercal \bm R
\, < \, s_{RF} \, < \, \frac{1}{\tau-1} \bm q_1^\intercal \bm R 
\, = \, \log(2) \approx 0.693.
\]
It was shown in \cite[Prop.~6]{nilsson-rfib} that 
\[
\# \vartheta^{m+1}(b) \, = \, \# \vartheta^{m}(a)
\, = \, (m+1) \prod_{j=2}^{m+1} (m+2-j)^{f_{j -2}},
\]
where $\{f_j\}_{j \in \mathbb{N}}$ denotes the Fibonacci sequence. From this, it is a straightforward calculation to check that any of the formulas for computing $s^I$ that are given in Theorem~\ref{Thm:entropy=inflation-entropy} indeed reproduces the expression for $s_{RF}$ in \eqref{Eq:s_RF}. \exend
\end{example}

With the example of the random Fibonacci substitution we have left the realm of constant length substitutions. Although Corollary~\ref{Corr:constant-length} is no longer applicable in this situation, there are other sufficient criteria to ensure the disjoint set condition. In \cite{rust}, two properties of $\vartheta$ are introduced, called \emph{disjoint images} and \emph{disjoint inflation images}. The \emph{disjoint inflation images} property can be shown to be equivalent to the disjoint set condition, whereas the \emph{disjoint images} property is stronger in general. We can thereby carry over a result  that shows that the disjoint set condition is a somewhat generic feature of a primitive, semi-compatible random substitution.

\begin{fact}\label{Fact:prefix-suffix}
Suppose\/ $\vartheta$ does not satisfy the disjoint set condition \eqref{Eq:upper-bound-condition}. Then, the following two properties hold.
\begin{enumerate}
\item There are letters\/ $a,b \in \mc A$ and\/ $u_a \in \vartheta(a)$, $u_b \in \vartheta(b)$ such that\/ $u_a$ is a prefix of\/ $u_b$.
\item There are letters\/ $a,b \in \mc A$ and\/ $u_a \in \vartheta(a)$, $u_b \in \vartheta(b)$ such that\/ $u_a$ is a suffix of\/ $u_b$.
\end{enumerate}
\end{fact}
The proof for the existence of a prefix was spelt out in \cite[Lem.~14]{rust}. The corresponding result for the suffix follows exactly the same line of argument. Corollary~\ref{Corr:constant-length} is obviously a special case of this result for the constant-length setting.

\begin{example}
Many examples that we have considered so far are actually \emph{compatible} in the sense that all deterministic marginalisations of the random substitution produce the same shift-space $\mathbb{X}$, compare \cite{bss}. We will now turn to an example that violates compatibility but is still semi-compatible. This is a random variant of the square of the deterministic Fibonacci substitution, the entropy of which was treated in \cite{nilsson-sqfib}. It is determined by $\vartheta_{RF^2} \colon a \mapsto \{baa\}, \, b \mapsto \{ab,ba\}$, $\lambda = \tau^2$, $\bm R = \frac{1}{1+\tau} (\tau,1)^\intercal$. 
\\Since there is no level-$1$ inflation word that appears as the suffix of another, we can apply Fact~\ref{Fact:prefix-suffix} to conclude that the disjoint set condition holds and thus,
\[
s \, = \, \frac{1}{\tau^2 -1} \frac{1}{\tau+1} \log(2) 
\, = \, \frac{1}{\tau^3} \log(2),
\]
reproducing the result given in \cite[Thm.~2]{nilsson-sqfib}. This way, we can avoid any of the technical combinatorial estimates presented in \cite{nilsson-sqfib} to work out the entropy.
\exend
\end{example}

Let us now turn to some examples that have not yet been covered by the literature on entropy (to the best of the author's knowledge).

\begin{example}[Random paper folding]
Let $\vartheta_{RPF} \colon a \mapsto \{ab,ba\},\, b \mapsto \{cb,bc\},\, c \mapsto \{ad,da\}$ and $ d \mapsto \{cd,dc\}$ be a random substitution with $\lambda = 2$ and $R_i = 1/4$ for all $1 \leqslant i \leqslant 4$. Clearly, the disjoint set condition for constant-length substitutions applies and we find
$
s = \log(2).
$ \exend
\end{example}

\begin{example}\cite[Ex.~19]{rust}.
We consider the random substitution $\vartheta \colon a \mapsto \{ abbabba, ababbba\}$, $b \mapsto \{a \}$, $\lambda = 4$, $\bm R =\frac{1}{2} (1,1)^\intercal$.
This example was shown to satisfy \emph{global unique recognisability} \cite[Def.~18]{rust}, a property that precludes the existence of periodic points \cite[Prop.~21]{rust}. Since it also implies the disjoint set condition \cite[Prop.~23]{rust}, we can work out the topological entropy to be $s = \frac{1}{6} \log(2)$.  \exend
\end{example}

\section{Further notions related to topological entropy}
\label{Sec:related}
In this section, we will discuss two variants of topological entropy. It turns out that they describe essentially the same quantity although they suggest slightly different interpretations. The first concept is \emph{geometric entropy}, describing the exponential growth rate of the number of admitted patterns as their length is increased. Secondly, we will show that the topological entropy can be constructed from the growth rate of \emph{periodic} elements, provided there is at least one (and thus infinitely many) periodic elements in the subshift $\mathbb{X}$ that is associated with our random substitution.

For a geometric interpretation of words, we choose a vector $\bm \psi = (\psi_1, \ldots, \psi_n)^{\intercal}$, where $\psi_i > 0$ is to be interpreted as the length of the tile associated to $a_i \in \mc A$. Naturally, this can be extended to assigning a geometric length to any given (legal) word in the following way.
\begin{definition}
For $u \in \mc L$, let the \emph{geometric length} of $u$ be defined as $ \psi_u = \bm \psi^\intercal \Phi(u) $. 
\end{definition}

Let $\mc F^G(L)$ denote the set of geometric patterns of a given length $L \in \mathbb{R}$, built from words in $\mc L$. For our purposes (which are of purely combinatorial nature) it suffices to regard  $\mc F^G(L)$ as a set of words with a common geometric length
\[
\mc F^G(L) \, = \, \{ w \in \mc L \mid \psi_w = L \}.
\]
For most values of $L\in \mathbb{R}$, this set will be empty since $\{ \psi_w \mid w \in \mc L \}$ is only a countable set. Also, $\# \mc F^G(L)$ is not necessarily increasing in $L$, even when restricted to this subset of $\mathbb{R}$. There might be values of $L$ that, despite being large, correspond to a rare statistic of relative letter frequencies and are therefore obtained only by few words $w \in \mc L$. We therefore suggest the following notion of geometric entropy.
\[
s^G \, := \, \limsup_{L \rightarrow \infty} \frac{1}{L} \log (\# \mc F^G(L)).
\]

\begin{remark}
An alternative that allows one to keep the $\lim$ instead of the $\limsup$ would be to replace $\# \mc F^G(L)$ by a summation over an appropriate interval of lengths. For example, the quantity
\[
\lim_{L \rightarrow \infty} \frac{1}{L} \log \biggl( \sum_{L' \leqslant L} \# \mc F^G(L') \biggr),
\]
where the sum runs over all obtainable values of $L'$, yields the same value $s^G$. This is also a natural extension of the symbolic setting, since 
\[
s \, = \, \lim_{n \rightarrow \infty} \frac{1}{n} \log \biggl( \sum_{k = 1}^{n} \# \mc L_k  \biggr)
\]
follows easily from the observation $\# \mc L_n \leqslant \sum_{k = 1}^{n} \# \mc L_k \leqslant n \# \mc L_n$. \exend
\end{remark}

Because of the uniform existence of letter frequencies, there is a well-defined average tile length $\bm \psi^\intercal \bm R$ in the limit of large legal words. Conversely, $\varrho := (\bm \psi^\intercal \bm R )^{-1}$ gives the corresponding density of left endpoints of the tiles. Keeping this interpretation in mind, the following result is not surprising

\begin{prop}
For a primitive, semi-compatible random substitution with topological entropy $s$, the geometric entropy is given by\/ $s^G = \varrho s$.
\end{prop}

\begin{proof}
First, we show $s^G \leqslant \varrho s$. For $\varepsilon > 0$, let $m_0 \in \mathbb{N}$ be such that for all $u \in \mc L$ with $\abs{u} = m \geqslant m_0$, it is 
\begin{equation}
\left\vert \frac{\abs{u}_{a_i}}{m} - R_i \right\vert 
\, < \, \varepsilon,
\end{equation}
for all $1 \leqslant i \leqslant n$. Suppose such a $u$ is given. Then, we obtain for the geometric length of $u$,
\[
\psi_u \, = \, \sum_{i=1}^n \abs{u}_{a_i} \psi_i 
\, < \, m \bm \psi^\intercal \bm R + \varepsilon m \sum_{i=1}^n \psi_i 
\, = \, m \left(\bm \psi^\intercal \bm R + \varepsilon \norm{ \bm \psi}_1 \right),
\]
and analogously for the lower bound.
With the notation $\varrho^{\pm}_\varepsilon = (\bm \psi^\intercal \bm R \pm \varepsilon \norm{ \bm \psi }_1)^{-1}$, this implies for $L \geqslant L_0 = m_0/\varrho^+_\varepsilon$,
\[
\mc F^G(L) \, \subset \, \bigcup_{m= \ceil{\varrho^+_\varepsilon L}}^{\floor{\varrho^-_\varepsilon L}} \mc L_m.
\]
Thereby, we find
\begin{align*}
\limsup_{L \rightarrow \infty} \frac{1}{L} \log \left( \# \mc F^G(L) \right) 
& \, \leqslant \, \limsup_{L \rightarrow \infty} \frac{\ceil{\varrho^-_\varepsilon L}}{L} \frac{1}{\ceil{\varrho^-_\varepsilon L}} \log \Bigl( \bigl( \floor{\varrho^-_\varepsilon L} - \ceil{\varrho^+_\varepsilon L} \bigr) \# \mc L_{\floor{\varrho^-_\varepsilon L}} \Bigr)
\\ & \, = \, \varrho^-_\varepsilon s \, \xrightarrow{\ts \varepsilon \rightarrow 0 \ts} \, \varrho s.
\end{align*}
For the opposite inequality, let $a_i \in \mc A$ and consider the sequence $L_n = \bm \psi^\intercal \Phi(\vartheta^n(a_i))$, for $n \in \mathbb{N}$. Clearly, $L_n/\ell_{n,i} \rightarrow \bm \psi^\intercal \bm R$, for $n \to \infty$, by the existence of letter frequencies. Thus,
\[
s^G \, \geqslant \, \limsup_{n \rightarrow \infty} \frac{1}{L_n} \log(\# \mc F^G(L_n)) 
\, \geqslant \, \limsup_{n \rightarrow \infty} \frac{\varrho}{\ell_{n,i}} \log (\# \vartheta^n(a_i)) = \varrho s^I_i = \varrho s,
\]
where we have used $\vartheta^n(a_i) \subset \mc F^G(L_n)$ in the second inequality.
\end{proof}

\begin{remark}
The restriction to semi-compatible random substitutions also enables us to choose \emph{natural tile lengths}, encoded in the common left PF eigenvector $\bm L$. More precisely, taking $\bm \psi = \bm L$ allows us to interpret $\vartheta$ as a (random/multi-valued) geometric inflation rule, similar to the case of deterministic substitutions. In this case, we obtain $s^G = s$. \exend
\end{remark}

Next, we discuss two more interpretations of $s^G$, which show that it can also be interpreted as the \emph{topological} entropy of a properly chosen dynamical system. First, with slight abuse of notation, we define a \emph{roof function} $\psi \colon \mathbb{X} \to \mathbb{R}_{+}$ by $ \psi(x)= \psi_{x^{}_0}$ and set
\[
Y \, = \, \{ (x,s) \mid x \in \mathbb{X}, 0 \leqslant s \leqslant \psi(x) \} \subset \mathbb{X} \times \mathbb{R},
\]
where $\mathbb{R}$ is equipped with the standard topology, $\mathbb{X} \times \mathbb{R}$ with the product topology and $Y$ with the subspace topology. We define an equivalence relation on the space $\mathbb{X} \times \mathbb{R}$ via $(x,s + \psi(x)) \sim (\sigma(x),s)$, where $\sigma$ denotes the left shift on $\mathbb{X}$. For every point  $z \in \mathbb{X} \times \mathbb{R}$, there are at most two points in $Y$ equivalent to $z$, and exactly one in the space $\widehat{Y} := Y/\sim$. Thus, $\widehat{T}_t(x,s) = (x,s+t)$, modulo the equivalence explained above, gives a well-defined map on $\widehat{Y}$, for every $t \in \mathbb{R}$. Thus $\widehat{T} = \{\widehat{T}_t\}$ is a one-parameter transformation group on $\widehat{Y}$.
\begin{definition}
The continuous dynamical system $(\widehat{Y},\widehat{T})$, equipped with the quotient topology, is called the \emph{suspension} or \emph{special flow} of the dynamical system $(\mathbb{X},\sigma)$.
\end{definition}

For general background on special flows, we refer to \cite[Ch.~11]{cfs}.
By the variational principle, the topological entropy of a dynamical system is given by the supremum of the metric entropies over all invariant Borel probability measures, compare \cite{tagi-zade} for the case of continuous group actions. 
Let $\mu$ be any shift-invariant Borel probability measure on $\mathbb{X}$ and $m$ the Lebesgue measure on $\mathbb{R}$. We define 
\[
\widehat{\mu} \, = \, \frac{(\mu \otimes m) \vert_{\widehat{Y}}}{\int_{\mathbb{X}} \psi \, \mathrm{d}\mu},
\]
which is easily checked to be a probability measure on $\widehat{Y}$. Because we assumed $\mu$ to be shift-invariant we can specify the normalization constant even further;
\[
\int_{\mathbb{X}} \psi \, \mathrm{d}\mu 
\, = \, \lim_{n \to \infty} \frac{1}{n} \sum_{k = 0}^{n-1} \int_{\mathbb{X}} \psi \circ \sigma^k \, \mathrm{d} \mu 
\, = \, \bm \psi^\intercal \bm R,
\]
using Lebesgue's dominated convergence theorem and the (uniform) existence of letter frequencies in the last step. The fact that $\widehat{\mu}$ is $\widehat{T}$-invariant is classic \cite[Sec.~11.1]{cfs}. Due to a theorem by Abramov \cite{abramov}, the metric entropies of the systems $(\widehat{Y},\widehat{T},\widehat{\mu})$ and $(\mathbb{X},\sigma,\mu)$ are related via
\begin{equation}\label{Eq:abramov}
h(\widehat{T},\widehat{\mu}) \, = \, \frac{h(\sigma,\mu)}{\int_{\mathbb{X}} \psi \, \mathrm{d}\mu} 
\, = \, \varrho \, h(\sigma,\mu).
\end{equation}
All the ergodic probability measures on $\widehat{Y}$ are of the form $\widehat{\mu}$ for some ergodic probability measure $\mu$ on $\mathbb{X}$ (compare for example \cite{savchenko}). Using the variational principle in conjunction with \eqref{Eq:abramov}, we therefore arrive at the following statement.

\begin{coro}
Let\/ $(\mathbb{X},\sigma)$ be the dynamical system derived from some primitive semi-compatible random substitution\/ $\vartheta$ and\/ $(\widehat{Y}, \widehat{T})$ its suspension with roof function $\psi$ as above. Then, the topological entropy\/ $h(T)$ of\/ $(\widehat{Y}, \widehat{T})$ coincides with the geometric entropy\/ $s^G = \varrho s$. \qed
\end{coro}

\begin{remark}
Another, though closely related, approach starts from a geometric description of a tiling related to some $x \in \mathbb{X}$ via the set of the left endpoints of the tiles. This is given by
\[
G(x) \, := \, \{0\} \cup \{ \psi^{}_{x^{}_{[0,k]}} \mid k \in \mathbb{N}_0 \} \cup \{ - \psi^{}_{x^{}_{[-k,-1]}} \mid k \in \mathbb{N} \},
\]
if we take $0$ as a marker point for the initial position. It is easily seen that $G(x)$ is a Delone set of finite local complexity (compare \cite{baake}) for each $x \in \mathbb{X}$. All the sets $G(x)$ with $x \in \mathbb{X}$ can be embedded into an appropriate topological space of Delone sets $\mathcal{D}$; compare \cite{blr}. We define the geometric subshift associated to $\mathbb{X}$ as
\[
\mathbb{Y} \, = \, \overline{\{ G(x) + t \mid x \in \mathbb{X}, t \in \mathbb{R} \}},
\]
where the closure is taken with respect to the topology on $\mc D$. We have a group of translations $T = \{T_t\}$, with $T_t(G) = G + t$, for $G \in \mathbb{Y}$ and $t \in \mathbb{R}$, acting on $\mathbb{Y}$. It is straightforward to check that $(\mathbb{Y},T)$ and $(\widehat{Y},\widehat{T})$ are conjugate as topological dynamical systems. Their topological entropies therefore coincide. \\Finally, there is also the notion of a \emph{configurational} or \emph{patch counting} entropy corresponding to an arbitrary element with dense orbit in $\mathbb{Y}$, see \cite{lag-pleas}. (An element with dense orbit always exists; compare \cite{rs}.) This concept bears some resemblance to the definition of the geometric entropy $s^G$ given above. As was shown in \cite{blr}, the patch counting entropy coincides with the topological entropy of $\mathbb{Y}$ for the type of systems at hand and is therefore also given by $s^G$. \exend
\end{remark}

Let us now turn our attention to the subsets of periodic words. The class of general random substitutions is very large. It has been shown, for example, that every (topologically transitive) shift of finite type (SFT) can be obtained from an appropriate primitive random substitution \cite{grs}. For such an SFT (even for the more general class of sofic shifts) it is well-known that the topological entropy can be obtained from the growth rate of the number of periodic elements \cite[Thm.~4.3.6]{lind-marcus}. This raises the question whether a similar statement holds for the topological entropy of primitive semi-compatible random substitutions. 
\begin{definition}
Given a language $\mc L$, the \emph{set of periodic words of period $q$} is given by
\[
\mc P(q) \, = \, \{ u \in \mc L \mid \abs{u} = q, \; \text{with} \; u^N \in \mc L, \; \text{for all} \; N \in \mathbb{N}\},
\]
where $u^N = u \cdots u$ denotes the concatenation of $N$ times the word $u$. Note that these sets might be empty.
\end{definition}
Given a random substitution, the existence of periodic words in its language is a subtle problem. An initial investigation was performed in \cite{rust}, giving a number of criteria to exclude the existence of periodic points and an algorithm that checks whether a given word is periodic for a large family of semi-compatible random substitutions. However, there remain many cases that are not decidable by any of those results. The next proposition shows that if we \emph{assume} the existence of periodic words, there are sufficiently many to reproduce the full topological entropy. This is essentially due to the fact that periodic words produce periodic words under the substitution procedure.
\begin{prop}
The topological entropy of a primitive semi-compatible random substitution can be obtained from its sets of periodic words via
\[
s \, = \, \limsup_{q \rightarrow \infty} \frac{1}{q} \log (\# \mc P(q)),
\]
provided that there exists at least one (and thus infinitely many) periodic words.
\end{prop}

\begin{proof}
Choose some $q \in \mathbb{N}$ and $u \in \mc P(q)$. Then, $w \in \mc P(\abs{\vartheta^m(u)}_{(\ell)} )$ for all $w \in \vartheta^m(u)$, by construction. That is, $\vartheta^m(u) \subset \mc P(\abs{\vartheta^m(u)}_{(\ell)})$. Denote by $\Phi(u)$ the Abelianisation of $u$. It is then a straightforward application of PF theory to conclude that
\[
\lim_{m \rightarrow \infty} \frac{\abs{\vartheta^m(u)}_{(\ell)}}{\lambda^m} = \lim_{m \rightarrow \infty} \frac{1}{\lambda^m} \sum_{i=1}^n \abs{u}_{a_i} \abs{\vartheta^m(a_i)}_{(\ell)} = \lim_{m \rightarrow \infty}  \frac{1}{\lambda^m} \bm \ell_m^\intercal \Phi(u) 
=  \mathbf{1}^\intercal \bm R \bm L^\intercal \Phi(u) = \bm L^\intercal \Phi(u).
\]
Also, we find a lower bound for the cardinality of some sets of periodic words by
\[
\# \mc P \bigl(\abs{\vartheta^m(u)}_{(\ell)} \bigr) 
\, \geqslant \, \# \vartheta^m(u) 
\, = \, \prod_{i=1}^n (\# \vartheta^m(a_i))^{\Phi(u)_i}.
\]
Recalling that $q_{m,i} = \log(\# \vartheta^m(a_i))$, we find
\begin{align*}
\limsup_{m \rightarrow \infty} \frac{1}{\abs{\vartheta^m(u)}_{(\ell)}} \log \bigl(\# \mc P \bigl(\abs{\vartheta^m(u)}_{(\ell)}\bigr) \bigr) 
& \, \geqslant \, \limsup_{m \rightarrow \infty} \frac{1}{\abs{\vartheta^m(u)}_{(\ell)}} \bm q_m^\intercal \Phi(u) 
\\ & \, \geqslant \, \limsup_{m \rightarrow \infty} \frac{1}{\bm L^\intercal \Phi(u)} \frac{1}{\lambda^r} \bm q_r^\intercal \frac{M^{m-r}}{\lambda^{m-r}} \Phi(u) 
\\ & \, = \, \frac{1}{\lambda^r} \bm q_r^\intercal \bm R
\, \xrightarrow{ \ts r \rightarrow \infty \ts } \, s,
\end{align*}
where we have made use of Lemma~\ref{Lemma:q-bounds} in the last inequality and \eqref{Eq:Thm-entropy-calc} for the last step.
This finishes the proof.
\end{proof}

\section*{Acknowledgements} 

It is a pleasure to thank Michael Baake, Dan Rust, Neil Ma\~{n}ibo and Timo Spindeler for helpful discussions.
This work is supported by the German Research Foundation (DFG) via
the Collaborative Research Centre (CRC 1283) and by the Research Centre of Mathematical Modelling (RCM$^2$) of Bielefeld University.

\end{document}